\documentclass{article}
\usepackage[utf8]{inputenc}
\usepackage{amsmath}
\usepackage{amsfonts, amssymb}

\newtheorem{thm}{Theorem}

\newenvironment{proof}[1][] {\noindent {\bf Proof#1:} }{\hspace*{\fill}$\square$\medskip\par}

\def\E{{\mathbb E}\,}

\def\NN{{\mathcal N}}
\def\P{{\mathbb P}}

\def\SS{{\mathcal S}}

\def\MR{{\mathcal M}_R}

\title{On the maximum of a special random assignment process}
\author{
M.Lifshits\footnote{St.Petersburg State University. Postal address: Russia, 191023, St.Petersburg, Universitetskaia emb. 7/9. Emails: \texttt{mikhail@lifshits.org}, \texttt{tadevosiaan@yandex.ru}}, 
A.Tadevosian$^{*}$
}

\date{}

\begin{document}

\maketitle

\begin{abstract}
We consider the asymptotic behavior of the expectation of the maximum for a special assignment process
with constant or i.i.d. coefficients. We show how it depends on the coefficients' distribution.
\end{abstract}
\bigskip



\section{Introduction: the assignment problem}
Let $(w_{ij})_{1\le i,j\le n}$ be a matrix with non-negative entries, often called a cost matrix. 
For every $n$-permutation $\pi$ define a variable 
\[
    R_\pi = \sum_{i=1}^{n}w_{i\pi(i)}.
\]
The assignment problem (bipartite matching) consists in the study of 
$\max_{\pi\in\SS_n} R_\pi$ or $\min_{\pi\in\SS_n} R_\pi$. Here and elsewhere throughout the article
$\SS_n$ stands for the permutation group of degree $n$.

If the matrix $w$ is random, one deals with a random assignment process $(R_\pi)_{\pi\in\SS_n}$.

In general, when the random assignment problem is considered, one assumes that $(w_{ij})$ are independent identically distributed random variables and the dependence of the solution of their common distribution
and of the matrix size is studied, especially, for the case when the latter goes to infinity.

We recall here a number of the most important known results:
\begin{itemize}
    \item Uniform case, $U[0, 1]$, M\'ezard and Parisi \cite{ RandomLinkMatching}:
\[
    \E{\min\limits_{\pi\in \SS_n} R_\pi} = \zeta(2) -\frac{\zeta(2)/2 + 2\zeta(3)}{n} + O\left(\frac{1}{n^2}\right), \text{ as } n\to \infty,
\]
    \item Exponential case,  Exp(1), Aldous \cite{AldousZet2}:
\[
    \E{\min\limits_{\pi\in \SS_n} R_\pi} = \zeta(2) (1 + o(1)), \text{ as } n\to \infty,
\]
    \item Gaussian case, $\NN(0, 1)$, Mordant and Segers \cite{MordantSegers}, Lifshits and Tadevosian \cite{Lifshits2021GAP}:
\[
    \E{\max\limits_{\pi\in \SS_n} R_\pi} = n \sqrt{2 \log{n}} (1 + o(1)), \text{ as } n\to \infty.
\]
\end{itemize}

Recently, some rather general asymptotic results we obtained in \cite{ChengTkocz} and in \cite{Lifshits2022GenAP}.

Various extensions of the basic assignment problem described here are also available in the literature. 
It particular, one may consider for a rectangular cost matrix, see \cite{SorkinMN, BuckChan, CoppersmithSorkin, LinussonGeneralization, NairGeneralization, WaestlundEasyProof}. 
One can also define assignment problem as a combinatorial problem on more general graph structures such as spanning trees, Hamilton cycles, paths between two fixed vertices in complete $n$-graphs, see \cite{ChengTkocz}.
\medskip

The random assignment problem has a deep background and application in various fields of engineering and mathematics. The lines of past research vary from bipartite graph modeling of practical industrial 
problems to low-complexity algorithm design.  Bipartite matching problems emerge in auction
of goods, resource allocation in wireless networks and electronic vehicle charging, etc.

We recall just a few  real-world examples of problems that can be modeled by bipartite graphs. In  wireless communication scenarios, when a number of users are transmitting over a shared channel resource, we need to allocate the wireless channel as discrete resource blocks to the user population. The problem  of channel resource allocation in OFDMA 
(Orthogonal Frequency Division Multi-Access) system is modeled by Bai et al. \cite{Bai1} as a maximal matching of random bipartite graph. 
Later, Bai et al. \cite{Bai2} discussed the outage exponent for non-asymptotic sub-channel allocation problems, where the outage, i.e. failure in data transmission, happens when a user does not receive sufficient channel resource. 
A similar problem has been investigated by Chen et al. \cite{Chen3}, where a joint design of resource block allocation, and assignment of modulation and coding schemes is considered.

Aside from wireless communications, bipartite matching is proposed as a natural model for smart transportation. Ke et al. \cite{Ke} investigated ride sourcing as an online bipartite matching problem, and proposed a two-stage framework for optimization.

Based on the fundamental works of Shannon \cite{Shannon} and Hartley \cite{Hartley} in information and communication theory,
the applied mathematicians \cite{Bai1,Bai2,Chen3} came, some times implicitly, to a special form of the random cost matrix, namely, $w_{ij} = \log{(1+\gamma_{ij} h_{ij})}$, where $h_{ij}$ are i.i.d. standard exponential random variables, $\gamma_{ij} \geq 0$.  We call the family $(R_\pi)_{\pi\in\SS_n}$ for such matrices {\it a special} assignment process.

In this paper, we are interested in the behaviour of $\E{\max_{\pi \in \SS_n} R_\pi}$ for a  randomized special assignment process where $\gamma_{ij}$ are considered as i.i.d. non-negative random variables also independent of  the array $(h_{ij})$. This randomized setting has a reasonable practical interpretation.
For example it applies to the modeling of such telecom systems  where the spatial locations of user equipment are heterogeneous (e.g., the distances between users and the base station vary a lot).
Then the probability distribution of gamma can be interpreted as a continuum version of the collection of large-scale fading gains when the number of users is very large.



\section{A general result}

Let $\gamma$ be a generic random variable equidistributed with each $\gamma_{ij}$. Let $h$ be a standard exponential random variable
independent of $\gamma$. Finally, denote $w:=\log(1+\gamma h)$ a generic random variable equidistributed with each $w_{ij}$.
Our main tool is the Laplace transform
\[
    \Lambda(\rho):=\E \exp\left(-\rho \gamma^{-1}\right).
\]
\begin{thm} For $p\in (0,1)$ let g(p) be defined as a solution of equation
\begin{equation} \label{e:g}
   p= \Lambda(e^r-1)
\end{equation}
with respect to the variable $r$. If the function $p\mapsto g(p)$ is slowly varying at zero, then 
\begin{equation} \label{e:main}
   \E \MR \sim n\, g(1/n), \qquad \textrm{as } n\to \infty.
\end{equation}
\end{thm}

\begin{proof}
By independence of $\gamma$ and $h$ we have
\[
   \P(\gamma h\ge \rho) = \P\left(h\ge \rho \gamma^{-1}\right) = \E \exp\left(-\rho \gamma^{-1}\right)
   = \Lambda(\rho).
\]
Hence,
\[
    \P( w \ge r) = \P(\log(1+\gamma h)\ge r) =  \P( \gamma h\ge e^r-1) =\Lambda(e^r-1).
\]
In other words, we found that $g$ is the {\it tail quantile function},
\[
   g(p) := \inf \big\{ r: \P( w \ge r)\le p \big\} =  \inf \big\{ r:  \Lambda(e^r-1)  \le p \big\}.
\]
Now we may apply the result of \cite{Lifshits2022GenAP} for the general i.i.d. case: as long as the tail quantile function is slowly varying, one has
\[
   \E \MR \sim n\, g(1/n), \qquad \textrm{as } n\to \infty.
\]
as required in the theorem's claim.
\end{proof}

\section{Some important special cases}\label{sec:special_cases}

\subsection{Constant $\gamma$}
Let $\gamma\equiv c$ be a constant. Then $\Lambda(\rho)=\exp(-\rho/c)$,
hence
\[
  \Lambda(e^r-1)=\exp\left(-(e^r-1)/c\right)
\]
and by solving equation
\[
  p = \exp\left(-(e^r-1)/c\right)
\]
in $r$ for given $p$ we find
\[
  g(p)=r = \log\left(1+c\,|\log p|\right) \sim \log |\log p|, \qquad \textrm{as } p\to 0.
\]
We obtain from \eqref{e:main} that
\begin{equation}\label{asymp_c}
   \E \MR \sim n\, \log\log n, \qquad \textrm{as } n\to \infty.
\end{equation}

\subsection{Exponential $\gamma$}
Let $\gamma$ be a standard exponential random variable. Then
\[
     \Lambda(\rho) = \int_0^\infty \exp\left(-\rho y^{-1} -y \right) dy
   \sim \rho^{1/4} \exp(-2\sqrt{\rho}) \sqrt{\pi}, \qquad \textrm{as } \rho\to\infty.
\]

Hence, $\log\Lambda(\rho)\sim -2 \sqrt{\rho}$, as $\rho\to\infty$. Solving equation
$\Lambda(e^r-1)=p$ or $-\log\Lambda(e^r-1)=|\log p|$,
we arrive at
\[
  2\sqrt{e^r-1} =|\log p| (1+o(1))
\]
or equivalently
\[
   e^r-1 = \frac{|\log p|^2}{4}\ (1+o(1)),
\]
which yields
\[
  g(p)=r=\log \left( \frac{|\log p|^2}{4}\right) +o(1) \sim 2\log|\log p|,
  \qquad\textrm{as } p\to 0.
\]
Finally, from \eqref{e:main} it follows that
\begin{equation}\label{asymp_e}
  \E \MR \sim 2 n \, \log\log n, \qquad \textrm{as } n \to \infty.
\end{equation}
which is two times larger than for the case of constant $\gamma$.

\subsection{Distributions of $\gamma$ with polynomial tails}

Let now $\gamma$ have a density $q$ such that
\[
   q(y) \sim a \, y^{-\alpha}, \qquad \textrm{as } y\to\infty,
\]
with $a>0$, $\alpha>1$. Then
\[
    \Lambda(\rho) = \int_0^\infty \exp\left(-\rho y^{-1}\right)\, q(y)\, dy
    \sim  b \, \rho^{-(\alpha-1)}, \qquad \textrm{as } \rho\to\infty,
\]
with $b:=a \, \Gamma(\alpha-1)$. 
Equation
\[
   p=\Lambda(e^r-1) \sim b \, (e^r-1)^{1-\alpha}
\]
yields
\begin{eqnarray*}
   e^r-1 &\sim&  \left( \frac{p}{b}\right)^{-1/(\alpha-1)},
\\
  g(p) &=& r \sim \frac{|\ln p|}{\alpha-1}, \qquad \textrm{as } p\to 0.
\end{eqnarray*}
We conclude that
\begin{equation}\label{asymp_pt}
   \E \MR \sim n\ \frac{\ln n}{\alpha-1}, \qquad \textrm{as } n\to \infty.
\end{equation}

\subsection{Uniform $\gamma$}

Let $\gamma$ be uniformly distributed on $[0,1]$. Then

\[
   \Lambda(\rho) = \int_0^1 \exp(-\rho y^{-1}) dy 
   \sim  \exp\left(-\rho\right)/\rho, \qquad \textrm{as } \rho\to\infty.
\]

Solving equation $\Lambda(e^r-1)=p$,
we arrive at
\[
   e^r-1 = |\log p| (1+o(1)),
\]
which yields
\begin{equation}\label{asymp_u}
  g(p)=r \sim \log|\log p|,
  \qquad\textrm{as } p\to 0.
\end{equation}
Finally, from \eqref{e:main} it follows that
\[
  \E \MR \sim n \, \log\log n,
\]
which is the same as for the case of constant $\gamma\equiv 1$. In other words, for the uniform case only the largest possible values count.
\bigskip


\section{Some numerical aspects}

For practical applications, it is important to understand at which range of $n$ one may trust to the general asymptotics \eqref{e:main} and, in special cases, to further expansions like \eqref{asymp_c},\eqref{asymp_e},\eqref{asymp_pt}, or \eqref{asymp_u}. 
We present some answers to this problem obtained by numerical simulations.

Notice that in the matter of numerical simulations, one has to distinguish the annealed and the quenched cases.
Let $\mathcal{P}$ be the common distribution of $\gamma_{ij}$. In the annealed setting for the
assignment problem we deal with the double expectation 
\[
    \E \MR = \mathbb{E}_{\substack{ h\sim\mathrm{Exp}(1)\\\gamma\sim\mathcal{P}}}\,\,\MR.
\]

The simulation was performed as follows: for each $n$ from some set of positive integers  $I$ 
we generated  $m$ independent cost matrices $(w_{ij})_{1\le i,j\le n}$ and for each matrix solved the assignment problem exactly via Hungarian algorithm \cite{Hungarian}. 
When $m$ is large, the empirical mean of the optimal matching results provides a sufficiently sharp estimate of $\E \MR$. 





Then we compared the results obtained by Hungarian method  with our theoretical ones 
obtained from \eqref{e:g} and from further asymptotic expansions.

As a typical example, consider the results of numerical simulations in the  case of standard exponential $\gamma$ ran on $I = \{10 k \mid k = 1, \dots, 100\}$ and $m = 300$.




The estimate given by \eqref{e:g} matches the empirical mean of simulated data almost perfectly: the maximal relative error 4.53\% is attained at $n=10$ and quickly goes down to 0.25\% at $n=1000$.
Notice that equation \eqref{e:g} was solved through symbolical decomposing of the function 
$\Lambda(\cdot)$ up to the fourth asymptotic term.

On the contrary, using the ``beautiful"  one-term asymptotics $2 n \log{\log{n}}$ from \eqref{asymp_e}
does not provide the satisfactory results: in this way, the  maximum's expectation is overestimated by
$25$--$35\%$ with $33.4\%$ at $n=1000$. 



Other examples described in Section \ref{sec:special_cases}  provide  similar results. 
The common conclusion is that in the range of $n$ up to $1000$ the  numerical solution of \eqref{e:g} works well for approximation  of $\E \MR$ but the approximation via the main asymptotic term of $n g(1/n)$,  as  in 
\eqref{asymp_c},\eqref{asymp_e},\eqref{asymp_pt}, or \eqref{asymp_u}
is not precise enough. In general, the excessive relative error could be reduced by deriving few more
asymptotic terms for  $g(\cdot)$ where it is possible.

\medskip


The quenched setting for assignment problem  consists in sampling a matrix $\gamma$ 
from $\mathcal{P}$ once and forever  and then sampling matrices $h$ many times in order to evaluate
\[
    \E \MR = \mathbb{E}_{ h\sim\mathrm{Exp}(1)}\,\,\MR\quad\text{ for some }\, \gamma\,\text{ sampled from } \mathcal{P}.
\]
The numerical results obtained in the quenched setting are quite similar to those obtained in the annealed setting. 

\bibliographystyle{plain}

\begin{thebibliography}{99}

{\baselineskip=12pt

\bibitem{AldousZet2}
 D.J. Aldous, 2001. The {$\zeta(2)$} limit in the random assignment problem.
Random Structures \& Algorithms, 18, No. 4, 381--418.

\bibitem{SorkinMN}
 S.E. Alm,  G.B. Sorkin, 2002. Exact expectations and distributions for the random assignment problem.
Combinatorics, Probability and Computing, 11, No. 3, 217--248.

\bibitem{Bai1} 
B. Bai, W. Chen, Z. Cao, and K. B. Letaief, 2010. Max-matching diversity in OFDMA systems, IEEE
Transactions on Communications, vol. 58, no. 4, 1161--1171.

\bibitem{Bai2}  
B. Bai, W. Chen, K.B. Letaief, and Z. Cao, 2013. Outage exponent: A unified performance metric for
parallel fading channels, IEEE Transactions on Information Theory, vol. 59, no. 3, 1657--1677.

\bibitem{BuckChan}
 M.W. Buck, C.S. Chan, and D.P. Robbins, 2002. On the expected value of the minimum assignment.
Random Structures \& Algorithms, 21, No. 1, 33--58.

\bibitem{Chen3}  
Y. Chen, Y. Wu, Y.T. Hou, and W. Lou, 2021. mCore: Achieving sub-millisecond scheduling for 5G
MU-MIMO systems, in IEEE INFOCOM 2021 -- IEEE Conference on Computer Communications.

\bibitem{ChengTkocz}
 Y. Cheng, Y. Liu, T. Tkocz, and A. Xu, 2021 Typical values of extremal-weight combinatorial structures with independent symmetric weights.
Preprint.

\bibitem{CoppersmithSorkin}
 D. Coppersmith, G.B. Sorkin, 1999. Constructive bounds and exact expectations for the random assignment problem. Random Structures \& Algorithms, 15, No.2, 113--144.

\bibitem{Hartley} 
R.V.L. Hartley, 1928. Transmission of information.
Bell System Technical Journal,  7, 3, 535--563.

\bibitem{Ke}
J. Ke, F. Xiao, H. Yang, and  J. Ye, 2020.
Learning to delay in ride-sourcing systems: a multi-agent deep reinforcement learning framework,
IEEE Transactions on Knowledge and Data Engineering, 34, 5, 2280--2292.

\bibitem{Hungarian}
H. W. Kuhn, 1955. The Hungarian method for the assignment problem. Naval Research Logistics Quarterly, 2(1-2), 83–97.

\bibitem{Lifshits2021GAP}
M. Lifshits, A. Tadevosian, 2021. Gaussian assignment process. 
Preprint, {\it https://arxiv.org/abs/2107.04883}.

\bibitem{Lifshits2022GenAP}
M. Lifshits, A. Tadevosian, 2022. On the maximum of random assignment process. 
To appear in Statist. Probab. Letters, 2021, article 109530.  
Preprint {\it https://arxiv.org/abs/2201.11390}.

\bibitem{LinussonGeneralization}
 S. Linusson, J. Wästlund, 2000. A generalization of the random assignment problem. 
 Preprint, {\it https://arxiv.org/abs/math/0006146}.

\bibitem{RandomLinkMatching} 
 M. M\'ezard, G. Parisi, 1987. On the solution of the random link matching problem. Journal de Physique, 48 (9),  1451--1459.

\bibitem{MordantSegers}
 G. Mordant, J. Segers, 2021, Maxima and near-maxima of a Gaussian random assignment field. Statist. Probab. Letters. 173. 109087. 10.1016/j.spl.2021.109087. 

\bibitem{NairGeneralization}
C. Nair, B. Prabhakar, and M. Sharma, 2005. Proofs of the Parisi and Coppersmith--Sorkin random assignment conjectures. Random Struct. Alg., 27: 413--444.

\bibitem{Shannon}
C.E. Shannon, 1948. A mathematical theory of communication.
Bell Systems Technical Journal, 27, 4, 623--656.

\bibitem{WaestlundEasyProof}
 J. Wästlund, 2009. An easy proof of the $\zeta(2)$ limit in the random assignment problem. Electron. Commun. Probab. 14 261--269, 2009.

}
\end{thebibliography}

\end{document}